\documentclass[a4paper,11pt,draft]{article}
\usepackage{amsmath,amssymb,enumerate,color}
\usepackage{amsthm,color}
\usepackage{txfonts}

\newcommand{\R}{\mathbb{R}}
\newcommand{\C}{\mathbb{C}}

\newcommand{\ep}{\varepsilon}
\newcommand{\pa}{\partial}

\DeclareMathOperator{\lifespan}{LifeSpan}
\newcommand{\lr}[1]{{}\langle{}#1{}\rangle{}}

\newtheorem{theorem}{Theorem}[section]
\newtheorem{lemma}[theorem]{Lemma}
\newtheorem{proposition}[theorem]{Proposition}

\theoremstyle{remark}
\newtheorem{remark}{Remark}[section]
\theoremstyle{definition}

\newtheorem{definition}{Definition}[section]

\numberwithin{equation}{section}

\makeatletter
\def\@cite#1#2{[{{\bfseries #1}\if@tempswa , #2\fi}]}
\makeatother                  
\begin{document}
\begin{center}
\Large{{\bf
Upper bound for lifespan 
of solutions to 
certain semilinear parabolic, 
dispersive and hyperbolic 
equations  
via 
a unified test function method 
}}
\end{center}

\vspace{5pt}

\begin{center}
Masahiro Ikeda%
\footnote{
Department of Mathematics, Faculty of Science and Technology, Keio University, 3-14-1 Hiyoshi, Kohoku-ku, Yokohama, 223-8522, Japan/Center for Advanced Intelligence Project, RIKEN, Japan, 
E-mail:\ {\tt masahiro.ikeda@keio.jp/masahiro.ikeda@riken.jp}}
and
Motohiro Sobajima%
\footnote{
Department of Mathematics, 
Faculty of Science and Technology, Tokyo University of Science,  
2641 Yamazaki, Noda-shi, Chiba, 278-8510, Japan,  
E-mail:\ {\tt msobajima1984@gmail.com}}
\end{center}

\newenvironment{summary}{\vspace{.5\baselineskip}\begin{list}{}{%
     \setlength{\baselineskip}{0.85\baselineskip}
     \setlength{\topsep}{0pt}
     \setlength{\leftmargin}{12mm}
     \setlength{\rightmargin}{12mm}
     \setlength{\listparindent}{0mm}
     \setlength{\itemindent}{\listparindent}
     \setlength{\parsep}{0pt}
     \item\relax}}{\end{list}\vspace{.5\baselineskip}}
\begin{summary}
{\footnotesize {\bf Abstract.}
This paper is concerned with the blowup phenomena for 
initial-boundary  value problem 
\begin{equation}\label{our}
\begin{cases}
\tau \pa_t^2 u(x,t)-\Delta u(x,t)+a(x)\pa_t u(x,t)=\lambda |u(x,t)|^p, 
& (x,t)\in \mathcal{C}_\Sigma \times (0,T),
\\
u(x,t)=0, 
& (x,t)\in \pa\mathcal{C}_\Sigma \times (0,T),
\\
u(x,0)=\ep f(x), 
& x\in \mathcal{C}_\Sigma,
\\
\tau \pa_t u(x,0)=\tau \ep g(x), 
& x\in \mathcal{C}_\Sigma,
\end{cases}
\end{equation}
where $\mathcal{C}_\Sigma$ is a cone-like domain in $\R^N$ ($N\geq 2$) 
defined as 
$
\mathcal{C}_\Sigma
=
{\rm int}\left\{r\omega\in \R^N\;;\;r\geq  0, \omega\in \Sigma \right\}
$
with a connected open set $\Sigma$ in $S^{N-1}$ with smooth boundary $\pa\Sigma$. 
If $N=1$, then we only consider two cases $\mathcal{C}_\Sigma=(0,\infty)$ 
and $\mathcal{C}_\Sigma=\R$. 
Here $a(x)$ is a non-zero coefficient of $\pa_tu$ which could be complex-valued and space-dependent, 
$\lambda\in \C$ is a fixed constant, and $\ep>0$ is a small parameter. The constants $\tau=0,1$ switch the parabolicity and 
hyperbolicity of the problem \eqref{our}. 
The result proposes a {\it unified treatment} of estimates for lifespan of 
solutions to \eqref{our} by test function method. 
The Fujita exponent $p=1+2/N$ appears as a threshold of 
blowup phenomena for small data
when $\mathcal{C}_{\Sigma}=\R^N$, but the case of cone-like domain with boundary 
the threshold changes and explicitly given via the first eigenvalue 
of corresponding Laplace--Beltrami operator with Dirichlet boundary condition
as in Levine--Meier \cite{LM89}.
}
\end{summary}



\section{Introduction}
In this paper, we discuss the blow-up phenomena 
of the following initial-boundary value problem 
\begin{equation}\label{ndw}
\begin{cases}
\tau \pa_t^2 u(x,t)-\Delta u(x,t)+a(x)\pa_t u(x,t)=\lambda |u(x,t)|^p, 
& (x,t)\in \mathcal{C}_\Sigma \times (0,T),
\\
u(x,t)=0, 
& (x,t)\in \pa\mathcal{C}_\Sigma \times (0,T),
\\
u(x,0)=\ep f(x), 
& x\in \mathcal{C}_\Sigma,
\\
\tau \pa_t u(x,0)=\tau \ep g(x), 
& x\in \mathcal{C}_\Sigma,
\end{cases}
\end{equation}
where $\mathcal{C}_\Sigma$ is a cone-like domain in $\R^N$ ($N\geq 2$) 
defined as 
\[
\mathcal{C}_\Sigma
:=
{\rm int}\left\{\rho \omega\in \R^N\;;\;\rho \geq 0, \omega\in \Sigma \right\}
\]
with a connected open set $\Sigma$ in $S^{N-1}$ with smooth boundary $\pa\Sigma$. 
If $N=1$, then we only consider two cases $\mathcal{C}_\Sigma=(0,\infty)$ 
and $\mathcal{C}_\Sigma=\R$. 
Here $a(x)$ is a non-zero coefficient of $\pa_tu$ satisfying 
\begin{align}\label{ass.a(x)}
|a(x)|\leq a_0\lr{x}^{-\alpha}
\end{align}
with $\alpha\in [0,1]$;
note that $a(x)$ could be complex-valued and space-dependent, 
$\lambda\in \C$ is a fixed constant, and $\ep>0$ is a small parameter. 
The initial data $(f,\tau g)$ at least belongs to the following class: 
\[
(f,\tau g)\in H^1_0(\mathcal{C}_{\Sigma})\cap L^2(\mathcal{C}_{\Sigma}), 
\quad 
(\tau g+af)|x|^{\gamma}\in L^1(\mathcal{C}_{\Sigma})
\]
for some $\gamma\geq 0$. 
Finally, the constants $\tau=0,1$ switch the parabolicity and 
hyperbolicity of the problem \eqref{ndw}. 

The problem \eqref{ndw} is 
a unified form of several partial differential equations
of parabolic, dispersive and hyperbolic type. For example, 
if $\tau=0$, $a(x)\equiv 1$, 
$\mathcal{C}_\Sigma=\R^N$ (that is, $\Sigma=S^{N-1}$),
and $\lambda=1$ with $f\geq 0$, 
then \eqref{ndw} becomes the usual nonlinear heat equation 
of Fujita type:
\begin{equation}\label{nheat}
\begin{cases}
\pa_t u(x,t)-\Delta u(x,t)=u(x,t)^p, 
& (x,t)\in \R^N \times (0,T),
\\
u(x,0)=\ep f(x), 
& x\in \R^N.
\end{cases}
\end{equation}
If $\tau=0$, $a(x)\equiv -i$, $\lambda\in \C\setminus\{0\}$ 
and 
$\mathcal{C}_\Sigma=\R^N$,
then \eqref{ndw} becomes 
the nonlinear Schr\"odinger equation 
without gauge invariance:
\begin{equation}\label{schrodinger}
\begin{cases}
i\pa_t u(x,t)+\Delta u(x,t)=-\lambda |u(x,t)|^p, 
& (x,t)\in \R^N \times (0,T),
\\
u(x,0)=\ep f(x), 
& x\in \R^N,
\end{cases}
\end{equation}
If $\tau=0$, $a(x)\equiv e^{-i\zeta}$, $\lambda=e^{i(\eta-\zeta)}$ 
and $\mathcal{C}_\Sigma=\R^N$,
then \eqref{ndw} becomes 
the complex Ginzburg--Landau equation without gauge invariance:
\begin{equation}\label{CGL}
\begin{cases}
\pa_t u(x,t)-e^{i\zeta}\Delta u(x,t)=e^{i \eta} |u(x,t)|^p, 
& (x,t)\in \R^N \times (0,T),
\\
u(x,0)=\ep f(x), 
& x\in \R^N,
\end{cases}
\end{equation}
Finally, if $\tau=1$, $a(x)> 0$, $\lambda=1$ 
and $\mathcal{C}_\Sigma=\R^N$, then \eqref{ndw} becomes 
the nonlinear wave equation with space-dependent damping 
\begin{equation}\label{dampedwave}
\begin{cases}
\pa_t^2 u(x,t)-\Delta u(x,t)+a(x)\pa_t u(x,t)=|u(x,t)|^p, 
& (x,t)\in \R^N \times (0,T),
\\
u(x,0)=\ep f(x), 
& x\in \R^N,
\\
\pa_t u(x,0)=\ep g(x), 
& x\in \R^N.
\end{cases}
\end{equation}
Moreover, we can treat 
halved space 
$\mathcal{C}_\Sigma=\R_+^{k}\times \R^{N-k}$ when 
we take 
$\Sigma=\{(\omega_j)\in S^{N-1};\omega_l>0\ (l=1,\ldots,k)\}$. 
Therefore the corn-like domain $C_{\Sigma}$ is a generalization 
of domains with scale-invariance (see also Levine--Meier \cite{LM89}). 

The study of blowup phenomena 
for solutions to the respective equations has a long history. 
For the semilinear heat equation \eqref{nheat}, the blowup solutions 
were found in Fujita \cite{Fujita66} when $p<1+\frac{2}{N}$;
the exponent $p_F=1+2/N$ is well-known as the ``Fujita exponent". 
Then in the critical case $p=p_F$
blowup phenomena were shown 
by Hayakawa \cite{Hayakawa73}, 
Sugitani \cite{Sugitani75} (including fractional Laplacian) and 
Kobayashi--Shirao--Tanaka \cite{KST77}.
The sharp upper and lower estimates for lifespan of solutions to \eqref{nheat}
was established in Lee--Ni \cite{LN92} 
by using the structure of the heat kernel and the maximum principle as 
\[
\lifespan(u)\sim 
\begin{cases}
\ep^{-(\frac{1}{p-1}-\frac{N}{2})^{-1}}, & \text{if}\ p<1+\frac{2}{N}, 
\\
\exp(C\ep^{-(p-1)}), & \text{if}\ p=1+\frac{2}{N}.
\end{cases}
\]
Later, the further profile of blowup solutions 
including sign-changing solutions 
are considered by many mathematicians 
(see, e.g., Mizoguchi--Yanagida \cite{MY96,MY98}, Fujishima--Ishige \cite{FI11,FI12}
and their references therein).  

For the semilinear Schr\"odinger equation without gauge invariance \eqref{schrodinger}, 
blowup phenomena are discovered by Ikeda--Wakasugi \cite{IW13} when $p\leq p_F=1+2/N$.
Later the estimates of lifespan of solutions to \eqref{schrodinger} 
was found in Fujiwara--Ozawa when $p<p_F$. 
The similar analysis in view of stochastic aspect can be found in 
Oh--Okamoto--Pocovnicu \cite{OOParxiv}. We have to remark that 
the estimates of lifespan in the critical case $p=p_F$ left as an 
open problem in $L^1$-initial data. 

For the complex Ginburg--Landau equation without gauge invariance \eqref{CGL}, 
blowup and lifespan of solutions to \eqref{CGL} 
in one-dimensional torus is studied by 
Ozawa--Yamazaki \cite{OY03}. Of course 
the complex Ginburg--Landau equation with nonlinear term 
$(\kappa+i\beta)|u|^{p-1}u$ (with gauge invariance) 
has been considered 
(see 
for existence, e.g., 
Ginible--Velo \cite{GV96}, 
Okazawa--Yokota \cite{OY02}, 
and 
for blowup phenomena, e.g., 
Masmoudi--Zaag \cite{MZ08}, 
Cazenave--Dickstein--Weissler \cite{CDW13} 
and 
Cazenave--Correia--Dickstein--Weissler \cite{CCDW15}). 


For the nonlinear damped wave equation without gauge invariance \eqref{dampedwave}, 
the blowup phenomena and estimates of the lifespan of solutions to 
\eqref{dampedwave} have been intensively studied for almost 20 years. 
First result should be Li--Zhou \cite{LZ95} and they proved 
blowup and upper bound of lifespan of solutions of \eqref{dampedwave} 
for $1<p\leq  p_F$ when $a(x)=1$ and $N=1,2$. 
Then the same question for the case of $N=3$ is answered by 
Nishihara \cite{Nishihara03Ibaraki}. 
For general, but subcritical case $1<p<p_F$ with $a(x)=1$, 
Todorova--Yordanov \cite{TY01} established blowup phenomena 
of \eqref{dampedwave} for arbitrary dimensions. 
In  the critical case $p=p_F$ for general dimensions 
Zhang \cite{Zhang01} obtained the same conclusion. 
Then the interest goes to the case of 
time-dependent or space-dependent damping. 
For time dependent case, we refer the study of 
Lin--Nishihara--Zhai \cite{LNZ12}, 
Ikeda--Wakasugi \cite{IW15} and Ikeda--Ogawa \cite{IO16} and the reference therein.
In the case of space-dependent damping, 
Ikehata--Todorova--Yordanov \cite{ITY09} found that 
the threshold for global existence of global solutions 
with small initial data and 
blowup for arbitrary small initial data 
for \eqref{dampedwave} with $a(x)\sim \lr{x}^{-\alpha}$ and $\alpha\in [0,1)$. 
We point out that the threshold shifts 
from the Fujita exponent $p_F$ to $p_F(\alpha)=1+\frac{2}{N-\alpha}$. 
Very recently, Lai--Zhou \cite{LZarxiv} succeeded in proving 
the sharp estimate of lifespan of solutions to \eqref{dampedwave} 
when $a(x)=1$ and $p=p_F$ by 
applying the consideration in \cite{LN92}. 

The similar study of respective problems for halved space 
$\R^{k}_+\times \R^{N-k}$ has been done separately 
in literature 
(see e.g., 
Meier \cite{Meier88,Meier90},
Levine--Meier \cite{LM89,LM90} 
and 
Ikehata \cite{Ikehata03,Ikehata03-2,Ikehata04}). 
In particular, Levine--Meier \cite{LM89,LM90} considered 
the nonlinear heat equation in $\mathcal{C}_{\Sigma}$ by using 
the explicit representation for heat kernel on the cone-like domain 
and found the corresponding threshold for blowup phenomena.

We would summarize the situation of study of blowup phenomena that 
although the detailed analysis has been done for respective equations 
in the respective cases, but many open problems are posed 
separately. 

The purpose of the present paper is 
to give a unified treatment for proving the upper bound of lifespan of solutions 
by using test function method
to the general problem \eqref{ndw} in cone-like domain 
including the all respective critical situations for respective equations. 
The crucial idea is mainly in the proof of Lemma \ref{key} (see also Remark \ref{rem:key} below).

The paper is organized as follows. 
In Section 2, we demonstrate our technique 
for simple three cases 
$\pa_tu -\Delta u= u^p$, 
$\pa_t^2u-\Delta u+\pa_tu = |u|^p$ 
and 
$i\pa_t u+\Delta u = |u|^p$ in $\R^N$
to explain what is a crucial view point in our argument. 
In Section 3, we state some basic fact 
of selfadjointness of the Laplacian on $\mathcal{C}_{\Sigma}$
endowed with Dirichlet boundary condition  
for treating linear equations of the respective equation, 
the solvability of \eqref{ndw} for each case $\tau=0$ and $\tau=1$ 
and prepare an important lemma via 
the unified test functions in the proof 
of the upper bound of lifespan. 
Then Section 4 is devoted to give main results 
of the present paper and to prove them. 

\section{Test function method for the simple cases in whole space}

The purpose of this section is to explain 
our test function method by using {\it well-understood} model. 
To do this, we begin with the following semilinear heat equation of Fujita type:
\begin{equation}\label{eq:nheat}
\begin{cases}
\pa_t u_\ep (x,t)-\Delta u_\ep (x,t)=u_\ep (x,t)^p, 
& (x,t)\in \R^N \times (0,T),
\\
u_\ep (x,0)=\ep f(x), 
& x\in \R^N.
\end{cases}
\end{equation}
Here we assume $f\in C_0^\infty(\R^N)$ and $f\geq 0$. 
In this case, by the standard argument for semilinear equations,
we can construct a unique local-in-time classical nonnegative solution $u_\ep$
of \eqref{eq:nheat}. Therefore we define the lifespan of solutions $u_\ep$ as follows:
\[
\lifespan{(u_\ep)}
=
\{T>0\;;\;\text{there exists a classical solution of \eqref{eq:nheat} in $[0,T)$}\}. 
\]
The following assertion was given by \cite{LN92}. 
\begin{proposition}
\label{prop.demo.heat}
Assume that $f\in C_0^\infty(\R^N)$, $f\geq 0$ and $f\not\equiv 0$. 
Let $u_\ep$ be the unique classical solution of \eqref{eq:nheat}. 
If $1<p\leq 1+\frac{2}{N}$, 
then $\lifespan{(u_\ep)}<\infty$. 
Moreover, $\lifespan{(u_\ep)}$ has the following upper bound:
there exist constants $\ep_0>0$ 
and $C\geq 1$ such that for every $\ep\in (0,\ep_0]$, 
\[
\lifespan(u_\ep)\leq  
\begin{cases}
C\ep^{-(\frac{1}{p-1}-\frac{N}{2})^{-1}}, & \text{if}\ 1<p<1+\frac{2}{N}, 
\\
\exp(C\ep^{-(p-1)}), & \text{if}\ p=1+\frac{2}{N}. 
\end{cases}
\]
\end{proposition}
\begin{proof}
Set $r_0:=\max\{|x|\;;\;x\in {\rm supp}\,f\}$. 
Without loss of generality, we may assume $R_0:=2r_0^2<T_\ep$. 
Put the following functions
\[
\eta(s)
\begin{cases}
=1& \text{if}\ s\in [0, 1/2]
\\
\text{is decreasing}& \text{if}\ s\in (1/2,1)
\\
=0 & \text{if}\ s\notin [1,\infty), 
\end{cases}
\quad
\eta^*(s)
=
\begin{cases}
0& \text{if}\ s\in [0, 1/2),
\\
\eta(s)& \text{if}\ s\in [1/2,\infty),
\end{cases}
\]
($\eta\in C^\infty([0,\infty))$) 
and for $R>0$, define the cut-off functions 
\[
\psi_R(x,t):=
\left[
\eta\left(\frac{|x|^2+t}{R}\right)
\right]^{2p'}
\quad 
\psi_R^*(x,t):=
\left[
\eta^*\left(\frac{|x|^2+t}{R}\right)
\right]^{2p'}. 
\]
Then by the equation in \eqref{eq:nheat}, we see from integration by parts that 
for every $R\in [R_0,T_\ep)$, 
\begin{align*}
\int_{\R^N}u_\ep(x,t)^p\psi_R(x,t)\,dx
&=
\frac{d}{dt}\int_{\R^N}u_\ep(x,t)\psi_R(x,t)\,dx
-\int_{\R^N}
\Big(u_\ep(x,t)\pa_t\psi_R(x,t)+\Delta u_\ep(x,t)\psi_R(x,t)\Big)\,dx
\\
&=
\frac{d}{dt}\int_{\R^N}u_\ep(x,t)\psi_R(x,t)\,dx
-\int_{\R^N}
u_\ep(x,t)\Big(\pa_t\psi_R(x,t)+\Delta \psi_R(x,t)\Big)\,dx.
\\
&\leq 
\frac{d}{dt}\int_{\R^N}u_\ep(x,t)\psi_R(x,t)\,dx
+\int_{\R^N}
u_\ep(x,t)
\left(\frac{C_1}{R}
+
\frac{C_2}{R^2}\right)[\psi_R^*(x,t)]^{\frac{1}{p}}\,dx.
\end{align*}
Then putting $C=C_1+C_2/R_0$ and integrating it over $(0,T_\ep)$, we have 
for every $R\in [R_0,T_\ep)$
\begin{align}
\nonumber
\ep \int_{\R^N}f(x)\,dx
+
\int_0^{T_\ep}\!
\int_{\R^N}u_\ep(x,t)^p\psi_R(x,t)\,dx\,dt
&\leq 
\frac{C}{R}
\int_0^{T_\ep}\!
\int_{\R^N}u_\ep(x,t)[\psi_R^*(x,t)]^{\frac{1}{p}}\,dx\,dt
\\
\label{eq:newineq}
&\leq 
CR^{-(\frac{1}{p-1}-\frac{N}{2})}
\int_0^{T_\ep}\!
\int_{\R^N}u_\ep(x,t)^p\psi_R^*(x,t)\,dx\,dt,
\end{align}
where we have used $\psi_R(\cdot,0)\equiv 1$ on ${\rm supp}\,f$. 
At this moment, we put new functions $y\in C(0,T_\ep)$ and $Y\in C^1(0,T_\ep)$ as follows: 
\begin{align}
Y(R):=\int_0^Ry(r)r^{-1}\,dr, \quad y(r):=\int_0^T\!\int_{\R^N}u(x,t)^p\psi_r^*(x,t)\,dx\,dt.
\end{align}
Then we have
\begin{align*}
\int_{0}^{R}
y(r)
r^{-1}\,dr
&=
\int_{0}^{R}
\left(
\int_0^{T_\ep}\!\int_{\R^N}
  u_\ep(x,t)^p
  \left[
  \eta^*\left(\frac{|x|^2+t}{r}\right)
  \right]^{2p'}
\,dx\,dt
\right)
r^{-1}\,dr
\\
&=
\int_0^{T_\ep}\!\int_{\R^N}
  u_\ep(x,t)^p
\left(
\int_{0}^{R}
  \left[
  \eta^*\left(\frac{|x|^2+t}{r}\right)
  \right]^{2p'}
r^{-1}\,dr
\right)
\,dx
\,dt
\\
&=
\int_0^{T_\ep}\!\int_{\R^N}
  u_\ep(x,t)^p
\left(
\int_{(|x|^2+t)/R}^{\infty}
  \left[
  \eta^*\left(s\right)
  \right]^{2p'}
s^{-1}\,ds
\right)
\,dx
\,dt.
\end{align*}
Noting that the inequality
\[
\int_{\sigma }^{\infty}
  \left[
  \eta^*\left(s\right)
  \right]^{2p'}
s^{-1}\,ds
\leq 
\log 2 
  \left[
  \eta(\sigma)
  \right]^{2p'},
  \quad \sigma\geq 0
\]
can be verified by the decreasing property of $\eta$, we deduce 
\[
Y(R)\leq \log 2\int_0^{T_\ep}\!\int_{\R^N}u_\ep(x,t)^p\psi_R(x,t)\,dx\,dt.
\]
By using the function $Y$, \ref{eq:newineq} can be rewritten by 
\[
\left(\ep+\frac{Y(R)}{\log 2}\right)^p
\leq 
CR^{-(\frac{1}{p-1}-\frac{N}{2})(p-1)+1}Y'(R), \quad R\in (R_0,T_\ep).
\]
Therefore we obtain
\begin{align*}
0\leq \limsup_{R\to T_\ep}\left(\ep+\frac{Y(R)}{\log 2}\right)^{1-p}
&\leq 
\left(\ep\|f\|_{L^1(\R^N)}+\frac{Y(R_0)}{\log 2}\right)^{1-p}
-\frac{p-1}{C\log 2}\int_{R_0}^{T_\ep}r^{(\frac{1}{p-1}-\frac{N}{2})(p-1)-1}\,dr
\\
&\leq 
\left(\ep\|f\|_{L^1(\R^N)}\right)^{1-p}
-\frac{p-1}{C\log 2}\int_{R_0}^{T_\ep}r^{(\frac{1}{p-1}-\frac{N}{2})(p-1)-1}\,dr.
\end{align*}
This implies the desired upper bound for $T_\ep$. 
\end{proof}
\begin{remark}
The crucial point is to introduce the function $Y$. 
In fact, the inequality including integral for $t$ 
enables us to treat as a differential inequality 
by virtue of the the use of $Y$. 
This consideration will be summarised in Lemma \ref{key} below. 
\end{remark}

This argument is also applicable to the semilinear problem 
of damped wave equation 
\begin{equation}\label{eq:dampedwave}
\begin{cases}
\pa_t^2 u_\ep (x,t)-\Delta u_\ep (x,t)+\pa_t u_\ep(x,t)=|u_\ep(x,t)|^p, 
& (x,t)\in \R^N \times (0,T),
\\
u_\ep(x,0)=\ep f(x), 
& x\in \R^N,
\\
\pa_t u_\ep(x,0)=\ep g(x), 
& x\in \R^N,
\end{cases}
\end{equation}
where we assume that $f,g\in C_0^\infty(\R^N)$ 
with 
\[
\int_{\R^N}f(x)+g(x)\,dx>0.
\]
In this case existence of weak solutions to \eqref{eq:dampedwave} 
is proved for $1<p<\infty$ when $N=1,2$ and $1<p<\frac{N+2}{N-2}$ when $N\geq 3$.  
The definition of lifespan is changed as follows: 
\[
\lifespan{(u_\ep)}
=
\{T>0\;;\;\text{there exists a weak solution of \eqref{eq:dampedwave} in $[0,T)$}\}. 
\]
As in the proof of Proposition \ref{prop.demo.heat}, we can find the following estimate 
\begin{align*}
\ep \int_{\R^N}f(x)+g(x)\,dx
+
\int_0^{T_\ep}\!
\int_{\R^N}|u_\ep(x,t)|^p\psi_R(x,t)\,dx\,dt
&\leq 
CR^{-(\frac{1}{p-1}-\frac{N}{2})}
\int_0^{T_\ep}\!
\int_{\R^N}|u_\ep(x,t)|^p\psi_R^*(x,t)\,dx\,dt
\end{align*}
for $R\in (R_0,T_\ep)$. Therefore by the use of the function $Y$, we can 
easily prove the upper bound of lifespan of $u_\ep$. 
\begin{proposition}
\label{prop.demo.damped}
Assume that $f,g\in C_0^\infty(\R^N)$, $f\geq 0$ and $\int_{\R^N}f(x)+g(x)\,dx>0$. 
Let $u_\ep$ be the unique weak solution of \eqref{eq:dampedwave}. 
If $1<p\leq 1+\frac{2}{N}$, 
then $\lifespan{(u_\ep)}<\infty$. 
Moreover, $\lifespan{(u_\ep)}$ has the following upper bound:
there exist constants $\ep_0>0$ 
and $C\geq 1$ such that for every $\ep\in (0,\ep_0]$, 
\[
\lifespan(u_\ep)\leq  
\begin{cases}
C\ep^{-(\frac{1}{p-1}-\frac{N}{2})^{-1}}, & \text{if}\ 1<p<1+\frac{2}{N}, 
\\
\exp(C\ep^{-(p-1)}), & \text{if}\ p=1+\frac{2}{N}. 
\end{cases}
\]
\end{proposition}
\begin{remark}
The critical case $p=1+\frac{2}{N}$ of Proposition \ref{prop.demo.damped} has already been proved 
by Lai--Zhou \cite{LZarxiv}. 
It is worth noticing that the proof of 
Proposition \ref{prop.demo.damped} is much simpler than that of \cite{LZarxiv}. 
\end{remark}

Furthermore, by the same argument 
we can also treat semilinear Schr\"odinger equation 
without gauge invariance:
\begin{equation}\label{eq:schrodinger}
\begin{cases}
i\pa_t u_\ep(x,t)+\Delta u_\ep(x,t)=|u_\ep(x,t)|^p, 
& (x,t)\in \R^N \times (0,T),
\\
u(x,0)=\ep f(x), 
& x\in \R^N,
\end{cases}
\end{equation}
with $f\in C_0^\infty(\R^N)$. The existence of weak solutions to \eqref{eq:schrodinger} 
is proved for $1<p<\infty$ when $N=1,2$ and $1<p<\frac{N+2}{N-2}$ when $N\geq 3$.  
The definition of lifespan is changed as follows: 
\[
\lifespan{(u_\ep)}
=
\{T>0\;;\;\text{there exists a weak solution of \eqref{eq:schrodinger} in $[0,T)$}\}. 
\]
For simplicity, 
we suppose $if(x)\in [0,\infty)$ and $f\not\equiv0$. 
Then multiplying $\psi_R$ to the equation
and taking the real part, we have
\begin{align*}
\ep \int_{\R^N}if(x)\,dx
+
\int_0^{T_\ep}\!
\int_{\R^N}|u_\ep(x,t)|^p\psi_R(x,t)\,dx\,dt
&\leq 
CR^{-(\frac{1}{p-1}-\frac{N}{2})}
\int_0^{T_\ep}\!
\int_{\R^N}|u_\ep(x,t)|^p\psi_R^*(x,t)\,dx\,dt.
\end{align*}
This gives us the fact that the essential point is 
completely the same as the previous cases \eqref{eq:nheat}
and \eqref{eq:dampedwave}. Consequently, 
we can obtain the following assertion. 
\begin{proposition}
\label{prop.demo.schrodinger}
Assume that $f\in C_0^\infty(\R^N)$, $if(x)\in [0,\infty)$ and $f\not\equiv 0$. 
Let $u_\ep$ be the unique weak solution of \eqref{eq:schrodinger}. 
If $1<p\leq 1+\frac{2}{N}$, 
then $\lifespan{(u_\ep)}<\infty$. 
Moreover, $\lifespan{(u_\ep)}$ has the following upper bound:
there exist constants $\ep_0>0$ 
and $C\geq 1$ such that for every $\ep\in (0,\ep_0]$, 
\[
\lifespan(u_\ep)\leq  
\begin{cases}
C\ep^{-(\frac{1}{p-1}-\frac{N}{2})^{-1}}, & \text{if}\ 1<p<1+\frac{2}{N}, 
\\
\exp(C\ep^{-(p-1)}), & \text{if}\ p=1+\frac{2}{N}. 
\end{cases}
\]
\end{proposition}
\begin{remark}
The critical case $p=1+\frac{2}{N}$ of Proposition \ref{prop.demo.schrodinger} has not been 
proved so far. The assertion of \ref{prop.demo.schrodinger} 
can be regarded as a refinement of the 
results of Ikeda--Wakasugi \cite{IW13} and Fujiwara--Ozawa \cite{FO16}. 
It is worth noticing that 
the technique by Lai--Zhou \cite{LZarxiv} 
seems to be difficult to apply to \eqref{eq:schrodinger} 
because they use the positivity of 
heat kernel for heat equations. 
Despite of this difficultly, 
we could prove the blowup phenomena and lifespan estimates 
by using only the positivity of the nonlinear term. 
\end{remark}

\section{Preliminaries for general cases}

To generalize the argument in Section 2 into {\it certain problems 
in corn-like domains stated in the introduction}, 
we prepare some technical tools to 
indicate the existence of corresponding problems. 

\newcommand{\intc}{\int_{\mathcal{C}_\Sigma}}

\subsection{Corresponding linear equations in $\mathcal{C}_\Sigma$}

First we state the assertion for 
the first eigenvalue and eigenfunction of 
the Laplace--Beltrami operator in $\Sigma$ 
endowed with Dirichlet boundary condition
(see \cite[Chapter IX]{Vilenkin} for detail). 

\begin{lemma}\label{LB}
The Laplace--Beltrami operator $-\Delta_{\Sigma}$ 
in $L^2(\Sigma)$ endowed with domain 
$H^2(\Sigma)\cap H^1_0(\Sigma)$ is selfadjoint 
and all resolvent operator of $-\Delta_{\Sigma}$ are compact. 
The first eigenvalue $\lambda_\Sigma$ is nonnegative and simple, and 
the corresponding eigenfunction $\varphi_\Sigma$ is positive in $\Sigma$.
Moreover, $\lambda_\Sigma$ 
is positive if and only if $\Sigma\neq S^{N-1}$. 
\end{lemma}

\begin{remark}
In the case $\mathcal{C}_{\Sigma}=\R^k\times \R^{N-k}$, 
$\varphi_{\Sigma}$ and $\lambda_{\Sigma}$
are explicitly given by 
$\varphi_{\Sigma}(\omega)=\omega_1\omega_2\cdots \omega_k$ and $\lambda_{\Sigma}=k(N-2+k)$.
\end{remark}
Here we define $\gamma$ as a smallest root of the following: 
\[
\gamma=
\begin{cases}
0 & \text{if}\ N=1, \mathcal{C}_{\Sigma}=\R
\\ 
1 & \text{if}\ N=1, \mathcal{C}_{\Sigma}=\R_+
\\
\text{the positive root of }\gamma^2+(N-2)\gamma-\lambda_{\Sigma}=0
& \text{if}\ N\geq 2.
\end{cases}
\]
Then the following assertion holds. 
\begin{lemma}\label{harmonics}
Set 
\[
\Phi(x)=
\begin{cases}
0 & \text{if}\ N=1, \mathcal{C}_{\Sigma}=\R
\\ 
x & \text{if}\ N=1, \mathcal{C}_{\Sigma}=\R_+
\\
|x|^{\gamma}\varphi_\Sigma\left(\dfrac{x}{|x|}\right), \quad x\in \mathcal{C}_{\Sigma}.
\end{cases}
\]
Then $\Phi$ satisfies 
\[
\begin{cases}
\Delta \Phi(x)=0
&x\in \mathcal{C}_{\Sigma},
\\
\Phi(x)>0
&x\in \mathcal{C}_{\Sigma},
\\
\Phi(x)=0
&x\in \pa\mathcal{C}_{\Sigma},
\\
x\cdot\nabla \Phi(x)=\gamma\Phi(x)
&x\in \mathcal{C}_{\Sigma}.
\end{cases}
\]
\end{lemma}
\begin{remark}
In the case $\mathcal{C}_{\Sigma}=\R^k\times \R^{N-k}$, 
we can easily see that $\Phi(x)=x_1x_2\cdots x_k$. 
\end{remark}
\begin{proof}[Proof of Lemma \ref{harmonics}]
By Lemma \ref{LB}
we can directly verify the desired properties of $\Phi$. 
\end{proof}

Next we consider the properties of Dirichlet Laplacian in $\mathcal{C}_{\Sigma}$. 
Let $A_{\min}$ be defined as follows: 
\begin{align*}
\begin{cases}
A_{\min}u=-\Delta u, 
\\
D(A_{\min})=\{u\in C^\infty_0(\R^N\setminus\{0\})\;;\;u=0\ \text{on}\ \pa\mathcal{C}_\Sigma\}.
\end{cases}
\end{align*}
We first prove the Hardy inequality in $\mathcal{C}_{\Sigma}$.
The idea is originated in Sobajima--Watanabe \cite{SW16}. 
\begin{lemma}\label{lem:hardy}
For every $u\in D(A_{\min})$, 
\begin{align}\label{eq:hardy}
\left(\frac{N-2}{2}+\gamma\right)^2
\int_{\mathcal{C}_{\Sigma}}
\frac{|u(x)|^2}{|x|^2}\,dx
\leq 
\int_{\mathcal{C}_{\Sigma}}
|\nabla u(x)|^2\,dx.
\end{align}
\end{lemma}
\begin{proof}
Since $u\in D(A_{\min})$ can be approximated 
by functions belonging to $C^\infty_0(\mathcal{C}_{\Sigma})$ 
in $H^1(\mathcal{C}_{\Sigma})$-topology, it suffices to show \eqref{eq:hardy} 
for $u\in C^\infty_0(\mathcal{C}_{\Sigma})$.

Let $u\in C^\infty_0(\mathcal{C}_{\Sigma})$ and set $Q(r\omega)=r^{-\frac{N-2}{2}}\varphi_{\Sigma}(\omega)$. 
Setting $v=Q^{-1}u\in C^\infty_0(\mathcal{C}_{\Sigma})$, we have 
\begin{align*}
\int_{\mathcal{C}_{\Sigma}}
|\nabla u(x)|^2
\,dx
&=
\int_{\mathcal{C}_{\Sigma}}
Q(x)^2|\nabla v(x)|^2
\,dx
+
2\int_{\mathcal{C}_{\Sigma}}
Q(x)\nabla Q(x)\cdot {\rm Re}(\overline{v}(x)\nabla v(x))
\,dx
+
\int_{\mathcal{C}_{\Sigma}}
|\nabla Q(x)|^2|v(x)|^2
\,dx
\\
&=
\int_{\mathcal{C}_{\Sigma}}
Q(x)^2|\nabla v(x)|^2
\,dx
-
\int_{\mathcal{C}_{\Sigma}}
Q\Delta Q(x)|v(x)|^2
\,dx,
\end{align*}
where we used integration by parts for the second term. Noting that 
\begin{align*}
\Delta Q(r\omega)
&=
-\left(\frac{N-2}{2}\right)^2r^{-\frac{N-2}{2}-2}
\varphi_{\Sigma}(\omega)
+
r^{-\frac{N-2}{2}-2}\Delta_{\Sigma}\varphi_{\Sigma}(\omega)
\\
&=
-
\left[
\left(\frac{N-2}{2}\right)^2
+\lambda_{\Sigma}
\right]
r^{-\frac{N-2}{2}-2}\varphi_{\Sigma}(\omega)
\\
&=
-
\left(\frac{N-2}{2}+\gamma\right)^2
\frac{Q(x)}{r^2}, 
\end{align*}
we obtain \eqref{eq:hardy}. 
\end{proof}

Here we prove the essential selfadjointness 
of $A_{\min}$ under the condition $\gamma\geq \frac{4-N}{2}$. 
If $\Sigma=S^{N-1}$, then $\mathcal{C}_{\Sigma}=\R^N$ and $\gamma=1$. 
Therefore this condition becomes $N\geq 4$ which is equivalent 
to that of essential selfadjointness of $-\Delta$ with domain 
$C_0^\infty(\R^N\setminus\{0\})$ 
(see e.g., Reed--Simon \cite[Theorems X.11 and X.30]{RS2}). 
\begin{lemma}\label{lem:ess.sa}
If $\gamma\geq \frac{4-N}{2}$, then $A_{\min}$ is essentially selfadjoint in $L^2(\mathcal{C}_{\Sigma})$. 
\end{lemma}
\begin{proof}
To prove the essential selfadjointness of $A_{\min}$, 
it suffices to show that 
\begin{align}\label{eq:anih}
v\in L^2(\mathcal{C}_{\Sigma}), \quad 
\int_{\mathcal{C}_{\Sigma}}
(u+A_{\min}u)v
\,dx=0 \quad \forall u\in D(A_{\min})
\end{align}
implies $v=0$ a.e.\ on $\mathcal{C}_{\Sigma}$. Assume \eqref{eq:anih}. 
noting that since the operator $A_{\min}$ 
does not have 
pure imaginary coefficient, we may assume 
without loss of generality that $v$ is real. 
By elliptic regularity we have $C^\infty(\overline{\mathcal{C}_{\Sigma}}\setminus\{0\})$
and $v=0$ on $\pa\mathcal{C}_{\Sigma}\setminus\{0\}$. 
Then integration by parts yields 
\begin{align}\label{eq:anih2} 
\int_{\mathcal{C}_{\Sigma}}
uv+\nabla u\cdot \nabla v
\,dx=0 \quad \forall u\in D(A_{\min}).
\end{align}
Fix $\zeta\in C_0^\infty(\R)$ with $\zeta\equiv 1$ on $[-1,1]$. For $R>1$, set 
\[
\chi_R(x)=\frac{|x|}{\sqrt{1+|x|^2}}\zeta\left(\frac{\log |x|}{R}\right), 
\quad 
u(x)=[\chi_R(x)]^2v(x)\in D(A_{\min}). 
\]
Then \eqref{eq:anih2} can be rewritten by 
\begin{align*}
\int_{\mathcal{C}_{\Sigma}}
\zeta_R^2v^2
\,dx
+
\int_{\mathcal{C}_{\Sigma}}
|\nabla (\chi_Rv)|^2
\,dx
=
\int_{\mathcal{C}_{\Sigma}}
|\nabla \chi_R|^2v^2
\,dx.
\end{align*}
Using Lemma \ref{lem:hardy} and computing $\nabla \chi_R$ explicitly, we have 
\begin{align*}
&\int_{\mathcal{C}_{\Sigma}}
\left(
\frac{|x|^2}{1+|x|^2}+
\left(\frac{N-2}{2}+\gamma\right)^2
\frac{1}{1+|x|^2}
\right)
\zeta\left(\frac{\log |x|}{R}\right)^2
v^2
\,dx
\\
&\leq 
\int_{\mathcal{C}_{\Sigma}}
\left|
\frac{1}{(1+|x|^2)^{\frac{3}{2}}}\zeta\left(\frac{\log |x|}{R}\right)
+
\frac{1}{R(1+|x|^2)^{\frac{1}{2}}}\zeta'\left(\frac{\log |x|}{R}\right)
\right|^2
v^2
\,dx.
\end{align*}
Since all coefficient of $v^2$ are bounded and have limits, dominated convergence theorem 
gives 
\begin{align*}
&\int_{\mathcal{C}_{\Sigma}}
\left(
\frac{|x|^2}{1+|x|^2}+
\left(\frac{N-2}{2}+\gamma\right)^2
\frac{1}{1+|x|^2}
\right)
v^2
\,dx
\leq 
\int_{\mathcal{C}_{\Sigma}}
\frac{1}{(1+|x|^2)^{3}}
v^2
\,dx.
\end{align*}
Therefore by $\frac{N-2}{2}+\gamma\geq 1$ we obtain $v=0$ a.e.\ on $\mathcal{C}_{\Sigma}$. 
\end{proof}
In view of Lemma \ref{lem:ess.sa}, 
we denote $A$ as a closure of $A_{\min}$, that is, 
$A$ is selfadjoint in $L^2(\mathcal{C}_{\Sigma})$. 
Noting that form domain $D(A^{1/2})$ of $A$ coincides with $H_0^1(\mathcal{C}_{\Sigma})$, 
we see from the Gagliardo--Nirenberg--Sobolev inequalities that 

\begin{lemma}\label{embed}
Assume $\gamma\geq \frac{4-N}{2}$. Then 
$D(A^{1/2})$ is continuously embedded into 
\[
\begin{cases}
L^\infty(\mathcal{C}_{\Sigma})
& N=1, 
\\
L^q(\mathcal{C}_{\Sigma})\ \ (2<\forall q<\infty)
& N=2, 
\\
L^{\frac{2N}{N-2}}(\mathcal{C}_{\Sigma})
& N\geq 3.
\end{cases}
\]
\end{lemma}

Combining all lemmas as above, by the standard argument 
we obtain 
the wellposedness of local-in-time weak solutions 
to \eqref{ndw} with $\tau=0$ and $\tau = 1$. 
We omit both proof of propositions stated below.

\begin{proposition}\label{wellposed1}
Assume that $\tau=0$, $a(x)=e^{i\zeta}$, $\zeta\in [-\pi/2,\pi/2]$ 
and $\gamma\geq \frac{4-N}{2}$. Then 
for $1<p<\frac{N}{(N-2})_+$ and 
for $f\in H_0^1(\mathcal{C}_{\Sigma})$, 
there exist $T=T(\|f\|_{H^1(\mathcal{C}_{\Sigma})},\ep)>0$ 
and 
a unique weak solution $u$ of \eqref{ndw} in $[0,T)$
in the following sense: 
\[
u\in C([0,T);H_0^1(\mathcal{C}_{\Sigma}))
\cap 
L^p_{\rm loc}(\overline{\mathcal{C}_{\Sigma}}\times [0,T))
\]
with $u(x,0)=\ep f(x)$ and
for every $\psi\in C^1([0,T);D(A))$ with 
${\rm supp}\,\psi\subset\subset \overline{\mathcal{C}_{\Sigma}}\times [0,T)$ 
\begin{align*}
&
\ep 
e^{i\zeta}\int_{\mathcal{C}_{\Sigma}}
  f(x)\psi(x,0)
\,dx
+
\lambda
\int_0^T
\int_{\mathcal{C}_{\Sigma}}
  |u(x,t)|^p\psi(x,t)
\,dx
\,dt
\\
&
=\int_0^T
\int_{\mathcal{C}_{\Sigma}}
  \Big(\nabla u(x,t)\cdot\nabla \psi(x,t)
   -e^{i\zeta} u(x,t)\pa_t\psi(x,t)
  \Big)
\,dx
\,dt.
\end{align*}
\end{proposition}

\begin{proposition}\label{wellposed2}
Assume that $\tau=1$, $\lambda=1$ and $\gamma\geq \frac{4-N}{2}$. Then 
for $1<p<\frac{N}{N-2}$ and 
for $(f,g)\in H_0^1(\mathcal{C}_{\Sigma})\times L^2(\mathcal{C}_{\Sigma})$, 
there exist 
$T=T(\|f\|_{H^1(\mathcal{C}_{\Sigma})},\|g\|_{L^2(\mathcal{C}_{\Sigma})},\ep)>0$ and 
a unique weak solution $u$ of \eqref{ndw} in $[0,T)$
in the following sense: 
\[
u\in C([0,T);H_0^1(\mathcal{C}_{\Sigma}))
\cap 
C^1([0,T);L^2(\mathcal{C}_{\Sigma}))
\cap 
L^p_{\rm loc}(\overline{\mathcal{C}_{\Sigma}}\times [0,T))
\]
with $u(x,0)=\ep f(x)$ and
for every $\psi\in C^2([0,T);D(A))$ with 
${\rm supp}\,\psi\subset\subset \overline{\mathcal{C}_{\Sigma}}\times [0,T)$ 
\begin{align*}
&
\ep 
\int_{\mathcal{C}_{\Sigma}}
  g(x)\psi(x,0)  
\,dx
+
\int_0^T
\int_{\mathcal{C}_{\Sigma}}
  |u(x,t)|^p\psi(x,t)
\,dx
\,dt
\\
&=
\int_0^T
\int_{\mathcal{C}_{\Sigma}}
  \Big(
  \nabla u(x,t)\cdot\nabla \psi(x,t)
  -\pa_t u(x,t)\pa_t\psi(x,t)
  +a(x)\pa_t u(x,t)\psi(x,t)\Big)
\,dx
\,dt.
\end{align*}
\end{proposition}

To the end of this subsection we state 
the wellposedness of \eqref{ndw} with 
a singular damping coefficient $V_0|x|^{-1}$ ($V_0> 0$) in $\R^N$. 
The proof of following proposition is given in 
Ikeda--Sobajima \cite{IkedaSobajima1}.
\begin{proposition}
\label{prop:singular}
Let $N\geq 3$, $\tau=1$, $\lambda=1$, $\Sigma=S^{N-1}$, 
$a(x)=V_0|x|^{-1}$ $(V_0\geq 0)$ and 
\begin{align*}
\begin{cases}
1<p<\infty &\text{if}\ N=3,4
\\
1<p<\frac{N-2}{N-4} &\text{if}\ N\geq 5.
\end{cases}
\end{align*}
For every $(f,g)\in H^2(\R^N)\cap H^1(\R^N)$ and $\ep>0$, 
there exist $T=T(\|f\|_{H^2},\|g\|_{H^1},\ep)>0$ and 
a unique strong solution of \eqref{ndw} in the following class:
\[
u\in C^2([0,T];L^2(\R^N))\cap C^1([0,T];H^1(\R^N))\cap C([0,T];H^2(\R^N)).
\]
\end{proposition}

\subsection{The unified choice of test functions}

Although we already gave the same functions in Section 2, 
we repeat the argument for the reader's convenience. 

Here we fix two kinds of functions $\eta\in C^\infty([0,\infty))$ 
and $\eta^*\in L^\infty((0,\infty))$ 
as follows, which will be used in the 
cut-off functions:
\[
\eta(s)
\begin{cases}
=1& \text{if}\ s\in [0, 1/2]
\\
\text{is decreasing}& \text{if}\ s\in (1/2,1)
\\
=0 & \text{if}\ s\notin [1,\infty), 
\end{cases}
\quad
\eta^*(s)
=
\begin{cases}
0& \text{if}\ s\in [0, 1/2),
\\
\eta(s)& \text{if}\ s\in [1/2,\infty).
\end{cases}
\]
\begin{definition}\label{psi}
For $p>1$, we define for $R>0$, 
\begin{align*}
\psi_R(x,t)
&=
[\eta(s_R(x,t))]^{2p'}, 
\quad 
(x,t)\in \mathcal{C}_\Sigma\times [0,\infty),
\\
\psi_R^*(x,t)
&=
[\eta^*(s_R(x,t))]^{2p'}, 
\quad 
(x,t)\in \mathcal{C}_\Sigma\times [0,\infty).
\end{align*}
with 
\[
s_R(x,t)=R^{-1}
\left(\lr{x}^{2-\alpha}+
t\right),
\]
where $\alpha$ is the constant in \eqref{ass.a(x)}. 
We also set
\[
P(R)=\left\{(x,t)\in \mathcal{C}_{\Sigma}\times [0,\infty)\;;\;
\lr{x}^{2-\alpha}+
t\leq R\right\}.
\]
\end{definition}

\begin{lemma}
Let $\psi_R$ and $\psi_R^*$ be as in Definition \ref{psi}.
Then $\psi_R$ satisfies the following properties:
\begin{itemize}
\item[\bf (i)] If $(x,t)\in P(R/2)$, then $\psi_R(x,t)=1$, 
and if $(x,t)\notin P(R)$, then $\psi_R(x,t)=0$.
\item[\bf (ii)] 
There exists a positive constant $C_1$ such that 
for every $(x,t)\in P(R)$, 
\begin{align*}
|\pa_t \psi_R(x,t)|\leq C_1R^{-1}[\psi_R^*(x,t)]^{\frac{1}{p}}.
\end{align*}
\item[\bf (iii)] 
There exists a positive constant $C_2$ such that 
for every $(x,t)\in P(R)$, 
\begin{align*}
|\pa_t^2 \psi_R(x,t)|
\leq 
C_2R^{-2}[\psi_R^*(x,t)]^{\frac{1}{p}}.
\end{align*}
\item[\bf (iv)] 
There exists a positive constant $C_3$ such that 
for every $(x,t)\in P(R)$, 
\begin{align*}
|\Delta \psi_R(x,t)|\leq C_3R^{-1}\lr{x}^{-\alpha}
[\psi_R^*(x,t)]^{\frac{1}{p}}.
\end{align*}
\end{itemize}
\end{lemma}

\begin{proof}
In view of the definition of $\psi_R$ and $\psi_R^*$, 
the assertions are verified by direct calculations.
\end{proof}

\subsection{Key lemma for estimates of lifespan}

\begin{lemma}\label{key}
Let $\delta>0$, $C_0>0$, $R_1>0$, $\theta\geq 0$ 
and $0\leq w\in L^1_{\rm loc}([0,T);L^1(\mathcal{C}_\Sigma))$
for $T>R_1$. Assume that 
for every $R\in [R_1,T)$, 
\begin{align}\label{criterion}
\delta 
+
\iint_{P(R)}
  w(x,t)\psi_R(x,t)
\,dx\,dt
\leq 
C_0R^{-\frac{\theta}{p'}}
\left(
\iint_{P(R)}
  w(x,t)\psi_R^*(x,t)
\,dx\,dt
\right)^{\frac{1}{p}}.
\end{align}
Then $T$ has to be bounded above as follows:
\begin{align*}
T\leq 
\begin{cases}
\left(R_1^{(p-1)\theta}+(\log 2)C_0^p\theta \delta^{-(p-1)}\right)^{\frac{1}{(p-1)\theta}}
&\text{if}\ \theta>0,
\\[5pt]
\exp\left(\log R_1+(\log 2)(p-1)^{-1}C_0^p\delta^{-(p-1)}\right)
&\text{if}\ \theta=0.
\end{cases}
\end{align*}
\end{lemma}

Although the upper bound of $T$ for $\theta>0$ 
can be verified by the simple way via Young inequality, 
we give a proof different from that via 
a different view point. 
This view point enables us to treat 
not only the subcritical case $\theta>0$ but also the critical case $\theta=0$. 

\begin{proof}[Proof of Lemma \ref{key}]
%

We define 
\begin{align*}
y(r):=
\iint_{P(r)}
  w(x,t)\psi_r^*(x,t)
\,dx\,dt, \quad r\in (0,T), 
\end{align*}
Then we have 
\begin{align*}
\int_{0}^{R}
y(r)
r^{-1}\,dr
&=
\int_{0}^{R}
\left(
\iint_{P(R)}
  w(x,t)
  \left[
  \eta^*\left(s_r(x,t)\right)
  \right]^{2p'}
\,dx\,dt
\right)
r^{-1}\,dr
\\
&=
\iint_{P(R)}
  w(x,t)
\left(
\int_{0}^{R}
  \left[
  \eta^*\left(\frac{\lr{x}^{2-\alpha}+t}{r}\right)
  \right]^{2p'}
r^{-1}\,dr
\right)
\,dx
\,dt
\\
&=
\iint_{P(R)}
  w(x,t)
\left(
\int_{(|x|^2+t)/R}^{\infty}
  \left[
  \eta^*\left(s\right)
  \right]^{2p'}
s^{-1}\,ds
\right)
\,dx
\,dt.
\end{align*}
On the other hand, by the definition of $\eta$ and $\eta^*$, 
for every $\sigma\geq 1$, 
\begin{align*}
\int_{\sigma}^{\infty}
  \left[
  \eta^*\left(s\right)
  \right]^{2p'}
s^{-1}\,ds=0, \quad 
\end{align*}
and for every $\sigma\in (0,1)$
\begin{align*}
\int_{\sigma}^{\infty}
  \left[
  \eta^*\left(s\right)
  \right]^{2p'}
s^{-1}\,ds
&=
\int_{\max\{1/2,\sigma\}}^{1}
  \left[
  \eta\left(s\right)
  \right]^{2p'}
s^{-1}\,ds
\\
&=
  \left[
  \eta\left(\sigma\right)
  \right]^{2p'}
\int_{1/2}^{1}
s^{-1}\,ds
\\
&=
  (\log 2)\left[
  \eta\left(\sigma\right)
  \right]^{2p'},
\end{align*}
where we have used the non-increasing property of $\eta$. 
Therefore we deduce from \eqref{criterion} that 
for $R\in  (R_1,T)$, 
\begin{align*}
\delta +\frac{1}{\log 2}
\int_{0}^R y(r)r^{-1}\,dr
&\leq
\delta
+\iint_{P(R)}
  w\psi_R
\,dx\,dt
\\
&\leq 
C_0
R^{-\frac{\theta}{p'}}\left(
\iint_{P(R)}
  w\psi_R^*
\,dx\,dt
\right)^{\frac{1}{p}}
\\
&\leq 
C_0
R^{-\frac{\theta}{p'}}
\left(
y(R)
\right)^{\frac{1}{p}}.
\end{align*}
Taking 
\[
Y(R)=
\int_{0}^R y(r)r^{-1}\,dr, \quad \rho\in (R_1, T), 
\]
we have 
\[
\Big((\log 2)\delta+Y(R)\Big)^p
\leq (\log 2)^pC_1^pR^{1-(p-1)\theta}Y'(R). 
\]
Taking 
\[
Y(R)=Z\left(\int_{R_1}^R r^{(p-1)\theta-1}\,dr\right), 
\quad 
0<\rho<\rho_T=\int_{R_1}^T r^{(p-1)\theta-1}\,dr.
\]
This gives 
\begin{equation}
\label{odi}
\frac{d}{d\rho}\Big((\log 2)\delta+Z(\rho)\Big)^{1-p}
\leq -(p-1)(\log 2)^{-p}C_1^{-p}, 
\quad 
\rho\in (0,\rho_T).
\end{equation}
Integrating it over $[\rho_1,\rho_2]\subset (0,\rho_T)$, we have
\begin{align}\label{rho12}
\Big((\log 2)\delta+Z(\rho_2)\Big)^{1-p}
&\leq 
\Big((\log 2)\delta+Z(\rho_1)\Big)^{1-p}
-(p-1)(\log 2)^{-p}C_1^{-p}(\rho_2-\rho_1).
\end{align}
Then we obtain
\[
\rho_2
<
\rho_1+(p-1)^{-1}(\log 2)C_1^p\delta^{-(p-1)}.
\]
Letting $\rho_2 \uparrow \rho_T$ and $\rho_1 \downarrow 0$, 
we find 
\[
\int_{R_1}^T r^{(p-1)\theta-1}\,dr
\leq 
(p-1)^{-1}(\log 2)C_1^p\delta^{-(p-1)}.
\]
This is nothing but the desired upper bound of $T$.
\end{proof}
\begin{remark}\label{rem:key}
The crucial idea in the present paper is 
to regard the inequality \eqref{criterion} as 
a differential inequality of $y(r)$ or $Y(R)$ in the proof. 
This idea with the choice of cut off functions in Definition \ref{psi} 
enable us to treat not only the case $\theta>0$ but the critical case $\theta=0$.  
We can find not only the upper bound of $T$ (which will be the lifespan) but also 
a lower estimate for $Y(R)$ by using \eqref{rho12}. 
\end{remark}

\section{Blowup phenomena and upper bound of lifespan for several equations}

In this section we prove blowup phenomena 
for several equations which can be written by the form \eqref{ndw}.
To simplify the situation, we split 
the case of the problem with $\tau=0$ and that with $\tau=1$. 
\begin{definition}\label{def:lifespan}
We denote $\lifespan(u)$ as the maximal existence time of solutions to \eqref{ndw}
in the sense of Propositions \ref{wellposed1}, \ref{wellposed2} and \ref{prop:singular}, 
respectively. Namely, 
\[
\lifespan(u)=
\sup \{T>0\;;\;\text{$u$ is a unique weak (strong) solution of \eqref{ndw} in $[0,T)$}\}. 
\]
\end{definition}
The statements of the main results are the following: 
\begin{theorem}\label{main0}
Assume that $\tau=0$, $a(x)=e^{i\zeta}$, $\zeta\in [-\pi/2,\pi/2]$ 
and $\gamma\in \{0,1\}\cup[\frac{4-N}{2},\infty)$ and  
$1<p<\frac{N}{N-2}$. 
Let $u$ be the unique solution of 
\eqref{ndw} with $f\in H_0^1(\mathcal{C}_{\Sigma})$ satisfying $f \Phi\in L^1(\mathcal{C}_{\Sigma})$
and 
\[
\int_{\mathcal{C}_{\Sigma}}
  f(x)\Phi(x)
\,dx\notin\{-\rho\lambda e^{-i\zeta}\in \C\;;\;\rho\geq 0\}.
\]
If $1<p\leq 1+\frac{2}{N+\gamma}$, 
then $\lifespan(u)<\infty$. 
Moreover one has 
\begin{align*}
\lifespan(u)
\leq 
\begin{cases}
\exp\Big(C\ep^{-(p-1)}\Big)
&\text{if}\ p=1+\frac{2}{N+\gamma},
\\
C\ep^{-\left(\frac{1}{p-1}-\frac{N+\gamma}{2}\right)^{-1}}
&\text{if}\ 1<p<1+\frac{2}{N+\gamma}.
\end{cases}
\end{align*}
\end{theorem}
\begin{theorem}\label{main1}
Assume that $\tau=1$, $\lambda=1$ and $\gamma\in \{0,1\}\cup[\frac{4-N}{2},\infty)$. Then 
for $1<p<\frac{N}{N-2}$ and 
for $(f,g)\in H_0^1(\mathcal{C}_{\Sigma})\times L^2(\mathcal{C}_{\Sigma})$. 
Let $a(x)$ be real-valued with \eqref{ass.a(x)} 
and let $u$ be the unique solution of \eqref{ndw} 
in Propositions \ref{wellposed2}. 
Further assume that $g\Phi,f\Phi\in L^1(\mathcal{C}_{\Sigma})$ with 
\[
\int_{\mathcal{C}_{\Sigma}}
  \Big(g(x)+a(x)f(x)\Big)
\Phi(x)\,dx>0.
\]
If $1<p\leq 1+\frac{2}{N+\gamma}$, 
then $\lifespan(u)<\infty$. 
Moreover, there exists a constant $\ep_0>0$ such that 
for every $0<\ep\leq \ep_0$
\begin{align*}
\lifespan(u)
\leq 
\begin{cases}
\exp\Big(C\ep^{-(p-1)}\Big)
&\text{if}\ p=1+\frac{2}{N+\gamma-\alpha},
\\
C\ep^{-\frac{2-\alpha}{2}\left(\frac{1}{p-1}-\frac{N+\gamma-\alpha}{2}\right)^{-1}}
&\text{if}\ 1+\frac{\alpha}{N+\gamma-\alpha}<p<1+\frac{2}{N+\gamma-\alpha},
\\
C_\delta\ep^{-(p-1)-\delta}\ (\forall\delta>0)
&\text{if}\ p=1+\frac{2}{N+\gamma-\alpha},
\\
C\ep^{-(p-1)}
&\text{if}\ 1<p<1+\frac{2}{N+\gamma-\alpha}.
\end{cases}
\end{align*}
If $u$ be the unique solution of \eqref{ndw} in Proposition \ref{prop:singular} 
(for $N\geq 3$ and $a(x)=V_0|x|^{-1}$ in $\R^N$), then 
Moreover, there exists a constant $\ep_0>0$ such that 
for every $0<\ep\leq \ep_0$
\begin{align*}
\lifespan(u)
\leq 
\begin{cases}
\exp\Big(C\ep^{-(p-1)}\Big)
&\text{if}\ p=\frac{N+1}{N-1},
\\
C\ep^{-\frac{2-\alpha}{2}\left(\frac{1}{p-1}-\frac{N-1}{2}\right)^{-1}}
&\text{if}\ \frac{N}{N-1}<p<\frac{N+1}{N-1}.
\end{cases}
\end{align*}
\end{theorem}
\begin{remark}
Since 
the upper bounds in Theorem \ref{main0} with $\mathcal{C}_{\Sigma}=\R^N$
coincides with that in Lee--Ni \cite{LN92} when we consider 
the nonlinear heat equation of Fujita-type 
and 
the upper bounds in Theorem \ref{main1} with $\mathcal{C}_{\Sigma}=\R^N$
matches that in 
Li--Zhou \cite{LZ95}, Nishihara \cite{Nishihara03Ibaraki}
and also Lai--Zhou \cite{LZarxiv}. 
Moreover, Theorems \ref{main0} and \ref{main1} 
give the lifespan of solutions even when 
the equation in the cone-like domain $\mathcal{C}_{\Sigma}$ 
has a critical nonlinearity which depends on the shape of $\Sigma$. 
In particular, we could obtain the 
the lifespan of solutions to nonlinear Schr\"odinger equation in $\R^N$
with the critical nonlinearity $p=p_F$. 
\end{remark}
\subsection{Proof of Theorem \ref{main0}}
We remark that the solution 
$u\in 
C([0,T);H^1(\mathcal{C}_{\Sigma}))\cap 
C([0,T);H^1(\mathcal{C}_{\Sigma}))
$ satisfies
\begin{align*}
&
\ep 
e^{i\zeta}\int_{\mathcal{C}_{\Sigma}}
  f\psi(0)
\,dx
+
\lambda
\int_0^T
\int_{\mathcal{C}_{\Sigma}}
  |u(t)|^p\psi(t)
\,dx
\,dt
\\
&
=\int_0^T
\int_{\mathcal{C}_{\Sigma}}
  \Big(\nabla u(t)\cdot\nabla \psi(t)
   -e^{i\zeta} u(t)\pa_t\psi(t)
  \Big)
\,dx
\,dt.
\end{align*}
Fix $\xi\in (-\frac{\pi}{2},\frac{\pi}{2})$ such that 
\[
{\rm Im}
\left(
e^{i(\xi+\eta)}\lambda^{-1}\int_{\mathcal{C}_{\Sigma}}f(x)\Phi(x)\,dx
\right)>0.
\]
Multiplying $\mu=\lambda^{-1}e^{i\xi}$ 
with $\xi\in (-\pi/2,\pi/2)$, we see 
that  
\begin{align}
\nonumber
&\ep 
\lambda^{-1}e^{i(\zeta+\xi)}\int_{\mathcal{C}_{\Sigma}}
  f\psi(0)
\,dx
+
e^{i\xi}
\int_0^T
\int_{\mathcal{C}_{\Sigma}}
  |u(t)|^p\psi(t)
\,dx
\,dt
\\
\label{eq:psi}
&
=
-\mu 
\int_0^T
\int_{\mathcal{C}_{\Sigma}}
  u(t)\Big(\Delta \psi(t)
   +e^{i\zeta}\pa_t\psi(t)
  \Big)
\,dx
\,dt,
\end{align}
where we used integration by parts which is verified 
by the regularity of test function $\psi(s)\in D(A)$. 

Here we choose $\psi(x,t)=\Phi(x)\psi_R(x,t)$ with $\alpha=0$. 
Since $\psi(x,t)=0$ on 
$(\pa\mathcal{C}_{\Sigma}\setminus\{0\})\times (0,\infty)$ and 
\[
\Delta\psi(x,t)=2\nabla \Phi(x)\cdot\nabla\psi_R(x,t)
+\Phi(x)\Delta \psi_R(x,t)
\] is a compactly supported 
bounded function, this choice is reasonable. Noting that 
\begin{align*}
\lim_{R\to \infty}
\left(
\int_{\mathcal{C}_{\Sigma}}
  f(x)\Phi(x)\psi_R(x,0)
\,dx
\right)
=
\int_{\mathcal{C}_{\Sigma}}
  f(x)\Phi(x)
\,dx, 
\end{align*}
we can choose $R_0>0$ and $c_0>0$ such that for every $R\geq R_0$, 
\[
{\rm Re}
\left(
e^{i(\xi+\eta)}\lambda^{-1}\int_{\mathcal{C}_{\Sigma}}f(x)\Phi(x)\psi_R(x,0)\,dx
\right)\geq c_0>0.
\]
Now we assume $R_0<\lifespan(u)$. 
Taking real part of 
\eqref{eq:psi}, we have
for $R\in (R_0,\lifespan(u))$, 
\begin{align*}
c_0
\ep 
+
\cos\xi\iint_{P(R)}
  |u(t)|^p\Phi\psi_R(t)
\,dx
\,dt
&\leq 
\iint_{P(R)}
  |u(s)|\Big(|\Delta \psi(t)|
   +|\pa_t\psi(t)|
  \Big)
\,dx
\,dt
\\
&\leq 
\frac{C}{R}\iint_{P(R)}
  |u(t)|\Phi
  [\psi^*_R(t)]^{\frac{1}{p}}
\,dx
\,dt
\\
&\leq 
\frac{C}{R}
\left(
\iint_{P(R)}
  \Phi
\,dx
\,dt
\right)^{\frac{1}{p'}}
\left(
\iint_{P(R)}
  |u(t)|^p
  \Phi\psi^*_R(t)
\,dx
\,dt
\right)^{\frac{1}{p}}
\\
&\leq 
C'
R^{-\frac{\theta}{p'}}
\left(
\iint_{P(R)}
  |u(t)|^p\Phi
  \psi^*_R(t)
\,dx
\,dt
\right)^{\frac{1}{p}}
\end{align*}
with 
\[
\theta=\frac{1}{p-1}-\frac{N+\gamma}{2}.
\]
Therefore applying Lemma \ref{key} with $w=|u|^p\Phi$, 
we have the desired upper bound of $\lifespan(u)$. 
\qed
\subsection{Proof of Theorem \ref{main1}}
Note that the solution $u$ satisfies 
\[
u\in C([0,T);H_0^1(\mathcal{C}_{\Sigma}))
\cap 
C^1([0,T);L^2(\mathcal{C}_{\Sigma}))
\cap 
L^p_{\rm loc}(\overline{\mathcal{C}_{\Sigma}}\times [0,T))
\]
with $u(x,0)=\ep f(x)$ and
for every $\psi\in C^2([0,T);D(A))$ with 
${\rm supp}\,\psi\subset\subset \overline{\mathcal{C}_{\Sigma}}\times [0,T)$ 
\begin{align*}
&
\ep 
\int_{\mathcal{C}_{\Sigma}}
  g\psi(0)  
\,dx
+\int_0^T
\int_{\mathcal{C}_{\Sigma}}
  |u(t)|^p\psi(t)
\,dx
\,dt
\\&
=
\int_0^T
\int_{\mathcal{C}_{\Sigma}}
  \Big(
  \nabla u(t)\cdot\nabla \psi(t)
  -\pa_t u(t)\pa_t\psi(t)
  +a\pa_t u(t)\psi(t)\Big)
\,dx
\,dt.
\end{align*}
By integration by parts for variable $x$ and $t$, 
we have 
\begin{align*}
&\ep 
\left(
\int_{\mathcal{C}_{\Sigma}}
  g\psi(0) 
  -f\pa_t\psi(0)
  +af\psi(0)
\,dx
\right)
+
\int_0^T
\int_{\mathcal{C}_{\Sigma}}
  |u(x,t)|^p\psi(x,t)
\,dx
\,dt
\\
&=
\int_0^T
\int_{\mathcal{C}_{\Sigma}}
u(t)
\left(\pa_t^2\psi(t)-\Delta \psi(t)-a\pa_t\psi(t)\right)
\,dx
\,dt
\end{align*}
Here noting that 
\begin{align*}
\lim_{R\to\infty}
\left(
\int_{\mathcal{C}_{\Sigma}}
  \Big(g\psi_R(0) 
  -f\pa_t\psi_R(0)
  +af\psi_R(0)\Big)
\Phi\,dx
\right)
=
\int_{\mathcal{C}_{\Sigma}}
  \Big(g+af\Big)
\Phi\,dx>0, 
\end{align*}
Then 
we see that there exist $R_0>0$ and $c_0>0$ such that 
for every $R\geq R_0$,
\[
\int_{\mathcal{C}_{\Sigma}}
  \Big(g\psi_R(0) 
  -f\pa_t\psi_R(0)
  +af\psi_R(0)\Big)
\Phi\,dx\geq c_0.
\]
Now we assume that $\lifespan(u)>R_0$. 
Since $\Phi$ is independent of $t$, it follows from 
Lemmas \ref{harmonics} and \ref{psi} that 
\begin{align*}
&\pa_t^2(\Phi\psi_R)-\Delta (\Phi\psi_R)-\pa_t(a(x)\Phi\psi_R)
\\
&=
\Phi\pa_t^2\psi_R
-2\nabla \Phi\cdot\nabla \psi_R
-\Phi\Delta\psi_R
-a(x)\Phi\pa_t\psi_R
\\
&
\leq 
\frac{C_2}{R^2}\Phi[\psi_R^*]^{\frac{1}{p}} 
+
\frac{4p'}{R}\nabla\Phi\cdot x\lr{x}^{-\alpha}[\psi_R^*]^{\frac{1}{p}} 
+
\frac{C_3}{R}\lr{x}^{-\alpha}\Phi[\psi_R^*]^{\frac{1}{p}} 
+
\frac{C_1}{R}\lr{x}^{-\alpha}\Phi[\psi_R^*]^{\frac{1}{p}} 
\\
&\leq 
\left(
\frac{C_2}{R^2}
+
\frac{4p'\gamma+C_1+C_3}{R}\lr{x}^{-\alpha}
\right)\Phi[\psi_R^*]^{\frac{1}{p}}. 
\end{align*}
Therefore choosing the test function $\psi(\cdot,t)=\Phi(\cdot)\psi_R(\cdot,t)\in D(A)$ 
implies that 
\begin{align*}
&c_0\ep
+\iint_{P(R)}
   |u(t)|^p\Phi\psi_R(t)
\,dx
\,dt
\\
&\leq 
\iint_{P(R)}
   u(t)\Big(\pa_t^2(\Phi\psi_R(t))-\Delta (\Phi\psi_R(t))-\pa_t(a\Phi\psi_R(t))\Big)
\,dx
\,dt
\\
&\leq 
C_4
\iint_{P(R)}
u\left(
\frac{1}{R^2}
+
\frac{1}{R}\lr{x}^{-\alpha}
\right)\Phi[\psi_R^*(t)]^{\frac{1}{p}}
\,dx
\,dt
\\
&\leq 
\frac{C_4}{R}
\left(
\iint_{P(R)}
\left(
\frac{1}{R}
+
\lr{x}^{-\alpha}
\right)^{p'}\Phi
\,dx
\,dt
\right)^{\frac{1}{p'}}
\left(
\iint_{P(R)}
|u(t)|^p
\Phi\psi_R^*(t)
\,dx
\,dt
\right)^{\frac{1}{p}}. 
\end{align*}
Noting that for $\beta=0,\alpha$, 
\begin{align*}
\iint_{P(R)}
\lr{x}^{-\beta p'}
\Phi
\,dx\,dt
&\leq 
\int_{0}^{R}
\int_{B(0,R^{\frac{1}{2-\alpha}})}
\lr{x}^{-\beta p'}
\Phi
\,dx\,dt
\\
&=
\int_{\Sigma}\varphi_{\Sigma}(\omega)\,d\omega
\int_{0}^{R}
\,dt
\int_0^{R^{\frac{1}{2-\alpha}}}
(1+r^2)^{-\beta p'/2}
r^{N+\gamma-1}
\,dr
\\
&\leq 
\begin{cases}
C R^{1+\frac{N+\gamma -\beta p'}{2-\alpha}} 
& \text{if}\ p>1+\frac{\beta}{N+\gamma-\beta},
\\
C R\log R
& \text{if}\ p=1+\frac{\beta}{N+\gamma-\beta},
\\
C R
& \text{if}\ p<1+\frac{\beta}{N+\gamma-\beta},
\end{cases}
\end{align*}
we deduce
\begin{align*}
&c_0\ep
+\iint_{P(R)}
  |u(t)|^p\Phi\psi_R(t)
\,dx\,dt
\leq 
C_5
q(R)^{1/p'}
\left(
\iint_{P(R)}
  |u(t)|^p\Phi\psi_R^*(t)
\,dx\,dt
\right)^{\frac{1}{p}}
\end{align*}
with 
\[
q(R)=
\begin{cases}
R^{-\frac{2}{2-\alpha}(\frac{1}{p-1}-\frac{N+\gamma-\alpha}{2})}
& \text{if}\ p>1+\frac{\alpha}{N+\gamma-\alpha},
\\
R^{-\frac{1}{p-1}}(\log R)
& \text{if}\ p=1+\frac{\alpha}{N+\gamma-\alpha},
\\
R^{-\frac{1}{p-1}}
& \text{if}\ 1<p<1+\frac{\alpha}{N+\gamma-\alpha}.
\end{cases}
\]
Therefore applying Lemma \ref{key} with $w=|u|^p\Phi$, we have 
\begin{align*}
T_{\max}
\leq 
\begin{cases}
\exp\Big(C\ep^{-(p-1)}\Big)
&\text{if}\ p=1+\frac{2}{N+\gamma-\alpha},
\\
C\ep^{-\frac{2-\alpha}{2}\left(\frac{1}{p-1}-\frac{N+\gamma-\alpha}{2}\right)^{-1}}
&\text{if}\ 1+\frac{\alpha}{N+\gamma-\alpha}<p<1+\frac{2}{N+\gamma-\alpha},
\\
C_\delta\ep^{-(p-1)-\delta}\ (\forall\delta>0)
&\text{if}\ p=1+\frac{2}{N+\gamma-\alpha},
\\
C\ep^{-(p-1)}
&\text{if}\ 1<p<1+\frac{2}{N+\gamma-\alpha}.
\end{cases}
\end{align*}
The part of proof for the solution in Proposition \ref{wellposed2} is complete. 

Finally, we only give a comment for the proof of 
upper bound for solution in Proposition \ref{prop:singular}. 
If we consider the case $a(x)=V_0|x|^{-1}$ and $\mathcal{C}_{\Sigma}=\R^N$, 
that is, $\gamma=0$ and $\alpha=1$, 
then we can deduce the same upper bound for the lifespan of $u$ as above 
only when $\frac{N}{N-1}<p\leq \frac{N+1}{N-1}$. 
The crucial point for that restriction is due to the integrability of 
\[
\iint_{P(R)}
|x|^{-p'}
\,dx\,dt.
\]
The proof is complete. 
\qed

\subsection*{Acknowedgements}
This work is partially supported 
by Grant-in-Aid for Young Scientists Research (B) 
No.16K17619 
and 
by Grant-in-Aid for Young Scientists Research (B) 
No.15K17571.


\end{document}